\documentclass[american,english,number]{elsarticle}
\usepackage[T1]{fontenc}
\usepackage[latin9]{inputenc}
\usepackage{amsthm}
\usepackage{amsmath}
\usepackage{graphicx}
\usepackage{amssymb}
\usepackage{esint}

\makeatletter
\numberwithin{equation}{section}
\numberwithin{figure}{section}
  \theoremstyle{plain}
  \newtheorem*{thm*}{Theorem}
\theoremstyle{plain}
\newtheorem{thm}{Theorem}
  \theoremstyle{definition}
  \newtheorem{defn}[thm]{Definition}
  \theoremstyle{plain}
  \newtheorem{prop}[thm]{Proposition}

\journal{Bulletin des Sciences Math\'{e}matiques}

\usepackage{ae,aecompl} 

\usepackage{enumitem}

\renewenvironment{enumerate}{\begin{oldenumerate}[topsep=0pt]}{\end{oldenumerate}}

\let\citet=\cite

\makeatother

\usepackage{babel}

\begin{document}
\selectlanguage{american}%
\global\long\def\o{\mathbf{1}}

\global\long\def\epsilon{\varepsilon}

\global\long\def\Ko{\mathcal{K}_{0}^{n}}

\global\long\def\T{\mathcal{T}}

\global\long\def\phi{\varphi}

\global\long\def\phip{\varphi^{\circ}}

\global\long\def\R{\mathbb{R}}

\global\long\def\cvx{\textrm{Cvx}_{0}\left(\R^{+}\right)}

\global\long\def\cvxn{\textrm{Cvx}_{0}\left(\R^{n}\right)}

\title{Characterization of Self-Polar Convex Functions}

\selectlanguage{english}%

\author{Liran Rotem}

\ead{liranro1@post.tau.ac.il}

\address{School of Mathematical Sciences, Tel-Aviv University, Tel-Aviv 69978,
Israel}
\begin{abstract}
In a work by Artstein-Avidan and Milman the concept of polarity is
generalized from the class of convex bodies to the larger class of
convex functions. While the only self-polar convex body is the Euclidean
ball, it turns out that there are numerous self-polar convex functions.
In this work we give a complete characterization of all rotationally
invariant self-polar convex functions on $\R^{n}$. \end{abstract}
\maketitle
\selectlanguage{american}%

\section{Introduction}

One of the most important concepts in convex geometry is that of polarity.
Denote by $\Ko$ the family of all closed, convex sets $K\subseteq\R^{n}$
such that $0\in K$. If $K\in\Ko$ we define its polar (or dual) body
as\[
K^{\circ}=\left\{ x\in\R^{n}:\ \left\langle x,y\right\rangle \le1\text{ for every }y\in K\right\} \in\Ko,\]
 where $\left\langle \cdot,\cdot\right\rangle $ is the standard inner
product on $\R^{n}$. It is easy to see that polarity is order reversing,
which means that if $K_{1}\subseteq K_{2}$ then $K_{1}^{\circ}\supseteq K_{2}^{\circ}$.
Polarity is also an involution - for every $K\in\Ko$ we have $\left(K^{\circ}\right)^{\circ}=K$.
It turns out that these two properties characterize polarity uniquely,
as the next theorem shows:
\begin{thm*}
Let $n\ge2$ and $\T:\Ko\to\Ko$ satisfy:
\begin{enumerate}
\item $\T\T K=K$ for all $K$.
\item If $K_{1}\subseteq K_{2}$ then $\T K_{1}\supseteq\T K_{2}.$
\end{enumerate}
Then $\T K=B\left(K^{\circ}\right)$ where $B\in\text{GL}_{n}$ is
a symmetric linear transformation.
\end{thm*}
This theorem was proven by Artstein-Avidan and Milman in \citet{artstein-avidan_concept_2008},
but similar theorems on different classes of convex bodies were proven
earlier by Gruber in \citet{gruber_endomorphisms_1992} and by B\"{o}r\"{o}czky
and Schneider in \citet{boeroeczky_characterization_2008}.

\selectlanguage{english}%
When dealing with polarity, the Euclidean ball $D_{n}\subseteq\R^{n}$
often plays a special role. The fundamental result here is that $D_{n}$
is the only self-polar convex body: $D_{n}^{\circ}=D_{n}$, and a
very simple proof shows that $D_{n}$ is the only body with this property. 

One example of the importance of $D_{n}$ when dealing with polarity
is the famous Blaschke-Santal\'{o} inequality. It states that if
$K$ is a symmetric convex body (i.e. If $K=-K$), then \[
\left|K\right|\cdot\left|K^{\circ}\right|\le\left|D_{n}\right|\cdot\left|D_{n}^{\circ}\right|=\left|D_{n}\right|^{2}.\]
Here $\left|\cdot\right|$ denotes the Lebesgue volume, and equality
holds if and only if $K$ is a linear image of $D_{n}$. There exists
a generalized version of the inequality for non-symmetric bodies,
but we will not need it here. The interested reader may consult \citet{meyer_blaschke-santalo_1990}. 

In recent years there was a surge of interesting results concerning
generalizations of various concepts from the realm of convex bodies
to the realm of convex (or, equivalently, log-concave) functions.
Our main object of interest will be $\cvxn$, the class of all convex,
lower semicontinuous functions $\phi:\R^{n}\to[0,\infty]$ satisfying
$\phi(0)=0$. Notice that we have an order reversing embedding of
$\Ko$ into $\cvxn$, sending $K$ to \[
\o_{K}^{\infty}(x)=\begin{cases}
0 & x\in K\\
\infty & \text{otherwise.}\end{cases}\]
 A natural question is whether one can extend the concept of polarity
from $\Ko$ to $\cvxn$. The answer to this question is {}``yes'',
as the following theorem by Artstein-Avidan and Milman (\citet{artstein-avidan_hidden_2011})
shows:
\begin{thm*}
Let $n\ge2$ and $\T:\cvxn\to\cvxn$ satisfy:
\selectlanguage{american}%
\begin{enumerate}
\item $\T\T\phi=\phi$ for all $\phi$.
\item If $\phi_{1}\le\phi_{2}$ then $\T\phi_{1}\ge\T\phi_{2}$ (here and
after, $\phi_{1}\ge\phi_{2}$ means $\phi_{1}(x)\ge\phi_{2}(x)$ for
all $x$).
\end{enumerate}
Then there exists a symmetric linear transformation $B\in\text{GL}_{n}$
and $c>0$ such that either:
\selectlanguage{english}%
\begin{enumerate}
\item [(a)]$\T\phi=\phi^{\ast}\circ B$, where $\phi^{\ast}$ is the classic
Legendre transform of $\phi$, defined by\[
\phi^{\ast}(x)=\sup_{y\in\R^{n}}\left[\left\langle x,y\right\rangle -\phi(y)\right]\]

\end{enumerate}
or 
\begin{enumerate}
\item [(b)]$\T\phi=\left(c\cdot\phi^{\circ}\right)\circ B$, where $\phi^{\circ}$
is the new Polarity transform of $\phi$, defined by\[
\phi^{\circ}(x)=\begin{cases}
\sup_{\left\{ y\in\R^{n}:\ \phi(y)>0\right\} }\frac{\left\langle x,y\right\rangle -1}{\phi(y)} & x\in\left\{ \phi^{-1}(0)\right\} ^{\circ}\\
\infty & x\notin\left\{ \phi^{-1}(0)\right\} ^{\circ}.\end{cases}\]

\end{enumerate}
\end{thm*}
Even though we have two essentially different order reversing involutions,
only the polarity transform extends the classical notion of duality,
in the sense that\[
\left(\o_{K}^{\infty}\right)^{\circ}=\o_{K^{\circ}}^{\infty}.\]
 Therefore it makes sense to think about $\phi^{\circ}$ as the polar
function to $\phi$. 

Once we have extended the definition of polarity to convex functions,
we want to extend our theorems as well. A functional version of the
Blaschke-Santal\'{o} inequality was proven by Ball in \citet{ball_isometric_1986}:
It follows from his work that if $\phi\in\cvxn$ is an even function
(i.e. $\phi(x)=\phi\left(-x\right)$), then \[
\int_{\R^{n}}e^{-\phi(x)}dx\cdot\int_{\R^{n}}e^{-\phi^{\ast}(x)}dx\le\left(\int_{\R^{n}}e^{-\frac{\left|x\right|^{2}}{2}}dx\right)^{2}=\left(2\pi\right)^{n}.\]
Again, there is a generalization for the non-even case, proven by
Artstein-Avidan, Klartag and Milman in \citet{artstein-avidan_santalo_2004}.
In the same paper is it also shown that if $\phi$ is a maximizer
of the Santal\'{o} product, then, up to a linear transformation,
we must have $\phi=\frac{\left|x\right|^{2}}{2}$. Since one can easily
check that $\left(\frac{\left|x\right|^{2}}{2}\right)^{\ast}=\frac{\left|x\right|^{2}}{2}$
and $\frac{\left|x\right|^{2}}{2}$ is the only function with this
property, we get once again that the maximizer in the Santal\'{o}
inequality is the unique self-dual function.

Rather surprisingly, the above mentioned theorem seems to use the
{}``wrong'' notion of polarity. It would be interesting have an
analogous theorem for $\phi^{\circ}$, that is to find the maximizer
of \[
\int_{\R^{n}}e^{-\phi(x)}dx\cdot\int_{\R^{n}}e^{-\phi^{\circ}(x)}dx.\]
 Given the classical and the functional Santal\'{o} inequalities,
it makes sense to conjecture that the maximizer here will be self-polar
as well, that is $\phi=\phi^{\circ}$. An independent argument by
Artstein-Avidan (\citet{artstein-avidan_steiner_????}) proves that
the maximizer must be rotationally invariant, i.e. of the form $\phi(x)=\rho(\left|x\right|)$
for a convex function $\rho:[0,\infty)\to[0,\infty]$. Therefore we
are naturally led to the following question: what are all the self-polar,
rotationally invariant convex functions? 

In order to answer this question we first follow \citet{artstein-avidan_hidden_2011}
and observe that if $\phi(x)=\rho\left(\left|x\right|\right)$ then
$\phi^{\circ}(x)=\rho^{\circ}(\left|x\right|)$, where $\rho^{\circ}$
is the polarity transform of $\rho$ on the ray (see the next section
for an exact definition). Therefore it is enough to find all self-polar
functions $\phi:[0,\infty)\to[0,\infty]$. An infinite family of such
functions is easy to present: For every $1\le p\le\infty$ the function\[
\phi_{p}(x)=\sqrt{\frac{\left(p-1\right)^{p-1}}{p^{p}}}\cdot x^{p}\]
 is self-polar. Our main result in this paper is that there are, in
fact, many other 1-dimensional self-polar functions. Specifically,
Theorem \ref{thm:main-thm} provides a complete characterization of
self-polar functions on the ray.

Unfortunately, because the set of self-polar functions is so big,
finding the Santal\'{o} maximizer inside this set seems rather intractable
at the moment, and the original question we started with remains open.
Nevertheless, we believe the results presented here are of independent
interest, and might have applications in several directions.

\selectlanguage{american}%

\section{Self-polar functions on the ray }

Let \foreignlanguage{english}{$\cvx$ be the class of all convex,
lower semicontinuous functions $\phi:[0,\infty)\to[0,\infty]$ satisfying
$\phi(0)=0$}. For a function $\phi\in\cvx$, we define its polar
$\phip$ as\foreignlanguage{english}{\[
\phip(x)=\sup_{y>0}\frac{xy-1}{\phi(y)}.\]
}

\selectlanguage{english}%
The division in the definition is formal, in the sense that $\phi(y)$
may be equal to $0$. We remedy the situation by defining $"\frac{+}{0}=\infty"$
and $"\frac{-}{0}=0"$. Put differently, we define $\phi^{\circ}(x)=\infty$
whenever there exists a $y\in\R^{+}$ such that $\phi(y)=0$ and $xy-1>0$
(or, in other words, whenever $x\notin\left\{ \phi^{-1}(0)\right\} ^{\circ}$). 

Just like in the $n$-dimensional case, polarity on the ray is also
an order reversing involution. Our main goal is to characterize all
functions $\phi\in\cvx$ such that $\phi=\phip$. 
\begin{defn}
Denote by $F$ the concave function $F(x)=\sqrt{x^{2}-1}$ (defined
for $x\ge1$). For $1\le q<\infty$, we define \[
T_{q}=\left\{ \phi\in\cvx:\ \begin{array}{l}
\phi(x)\ge F(x)\text{ for all }x\ge1\\
\phi(q)=F(q)\end{array}\right\} .\]
In other words, $T_{q}$ is the set of functions which are tangent
to $F$ at $q$. For $q=\infty$ we define \[
T_{\infty}=\left\{ \phi\in\cvx:\ \begin{array}{l}
\phi(x)\ge F(x)\text{ for all }x\ge1\\
\lim_{x\to\infty}\left(\phi(x)-F(x)\right)=0\end{array}\right\} .\]

\end{defn}
The classes $T_{1}$ and $T_{\infty}$ will be exceptional and somewhat
trivial. In fact, let us define the following:
\begin{defn}
For $\beta\in\R$ we define a function $\ell_{\beta}\in\cvx$ by $\ell_{\beta}(x)=\beta x$.
In other words, $\ell_{\beta}$ is the line through the origin with
slope $\beta$.
\end{defn}
Using this definition, it is not hard to see that $T_{1}=\left\{ \o_{[0,1]}^{\infty}\right\} $,
and $T_{\infty}=\left\{ \ell_{1}\right\} $. For $1<q<\infty$, the
class $T_{q}$ is infinite.

Our first proposition will explain the importance of these classes
when characterizing self-\foreignlanguage{american}{polar} functions:
\begin{prop}
\label{pro:invariant-set}\end{prop}
\begin{enumerate}
\item If $\phi\in T_{q}$ for some $1\le q\le\infty$, then $\phip\in T_{q}$. 
\item If $\phi=\phip$, then $\phi\in T_{q}$ for some $1\le q\le\infty$. \end{enumerate}
\begin{proof}
(i) If $q=1,\infty$ this is trivial by the above comment. For $1<q<\infty$
define two convex functions $\psi_{L}$ and $\psi_{U}$ as\begin{eqnarray*}
\psi_{L} & = & \o_{\left[0,\frac{1}{q}\right]}^{\infty}\wedge\ell_{F'(q)}\\
\psi_{U} & = & \o_{[0,q]}^{\infty}\vee\ell_{F(q)/q.}\end{eqnarray*}
Here and after, $\vee$ and $\wedge$ will denote supremum and infimum
in the lattice $\cvx$. In other words, $\phi_{1}\vee\phi_{2}=\max\left(\phi_{1},\phi_{2}\right)$,
and $\phi_{1}\wedge\phi_{2}$ is the biggest function in $\cvx$ which
is smaller than $\min\left(\phi_{1},\phi_{2}\right)$. To illustrate
these definitions we plot the graphs of $\psi_{L}$ and $\psi_{U}$: 

\noindent \begin{center}
\includegraphics[scale=0.65]{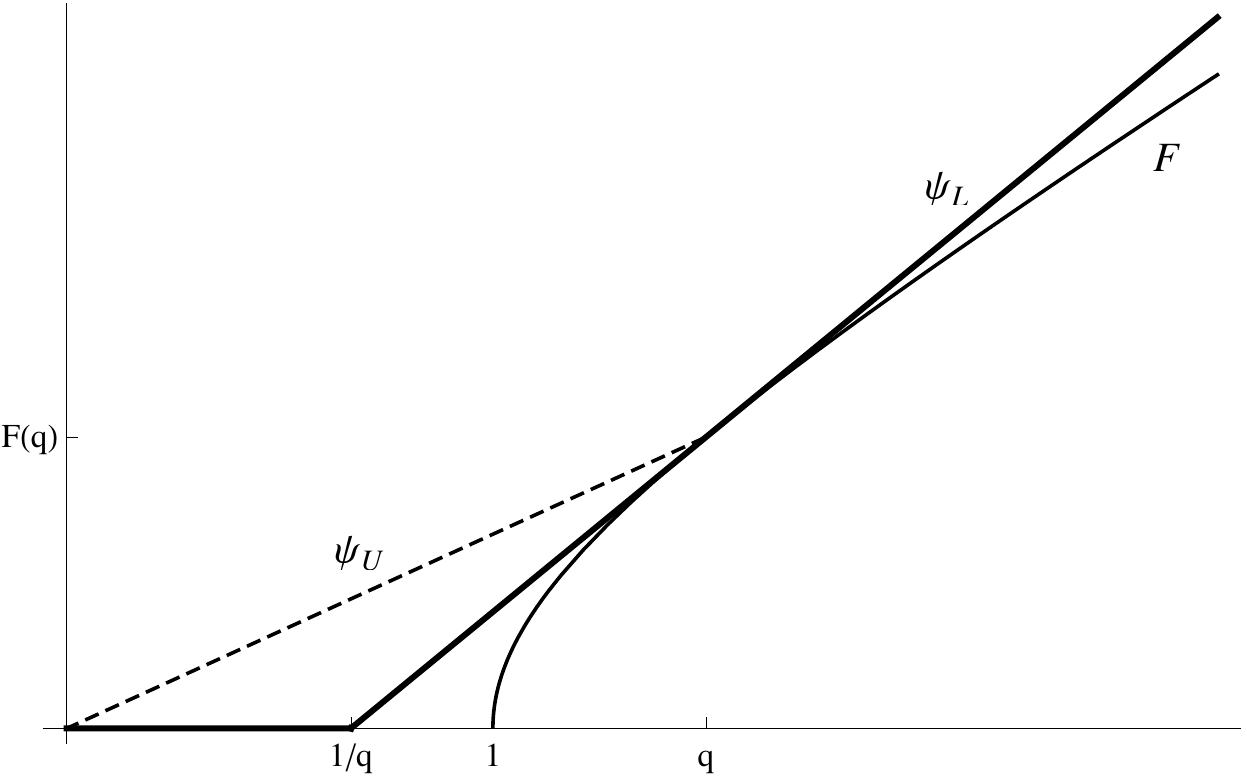}
\par\end{center}

It is clear that $\phi\in T_{q}$ if and only if $\psi_{L}\le\phi\le\psi_{U}$.
Since polarity is order reversing, we get that if $\phi\in T_{q}$
then $\psi_{U}^{\circ}\le\phip\le\psi_{L}^{\circ}$. But \[
\psi_{L}^{\circ}=\left(\o_{\left[0,\frac{1}{q}\right]}^{\infty}\right)^{\circ}\vee\left(\ell_{F'(q)}\right)^{\circ}=\o_{[0,q]}^{\infty}\vee\ell_{1/F'(q)}=\psi_{U},\]
 and thus $\psi_{U}^{\circ}=\psi_{L}$ , so $\phip\in T_{q}$ as well.

(ii) First notice that if $\phi=\phip$ then for every $x\ge0$ we
have\[
\phi(x)=\phip(x)=\sup_{y>0}\frac{xy-1}{\phi(y)}\ge\frac{x^{2}-1}{\phi(x)},\]
 and if $x\ge1$ this implies $\phi(x)\ge F(x)$. 

Define convex sets in $\left(\R^{+}\right)^{2}$ by \begin{eqnarray*}
C_{1} & = & epi(\phi)=\left\{ (x,y)\in\left(\R^{+}\right)^{2}:\ y\ge\phi(x)\right\} \\
C_{2} & = & hyp(F)=\left\{ (x,y)\in\left(\R^{+}\right)^{2}:\ y\le F(x)\right\} .\end{eqnarray*}
 If $d\left(C_{1},C_{2}\right)=0$ then $\phi\in T_{q}$ for some
$q$ (which can be $\infty$), so we will assume by contradiction
that $d(C_{1},C_{2})>0$. This means that there is a line $\ell$
strictly separating $C_{1}$ and $C_{2}$ (see, e.g. Theorem 11.4
in \citet{rockafellar_convex_1970}). Denote by $\beta$ the slope
of $\ell$ and by $a$ the intersection of $\ell$ and the $x-$axis.
Define $\psi=\o_{[0,a]}^{\infty}\wedge\ell_{\beta}$:

\noindent \begin{center}
\includegraphics[scale=0.65]{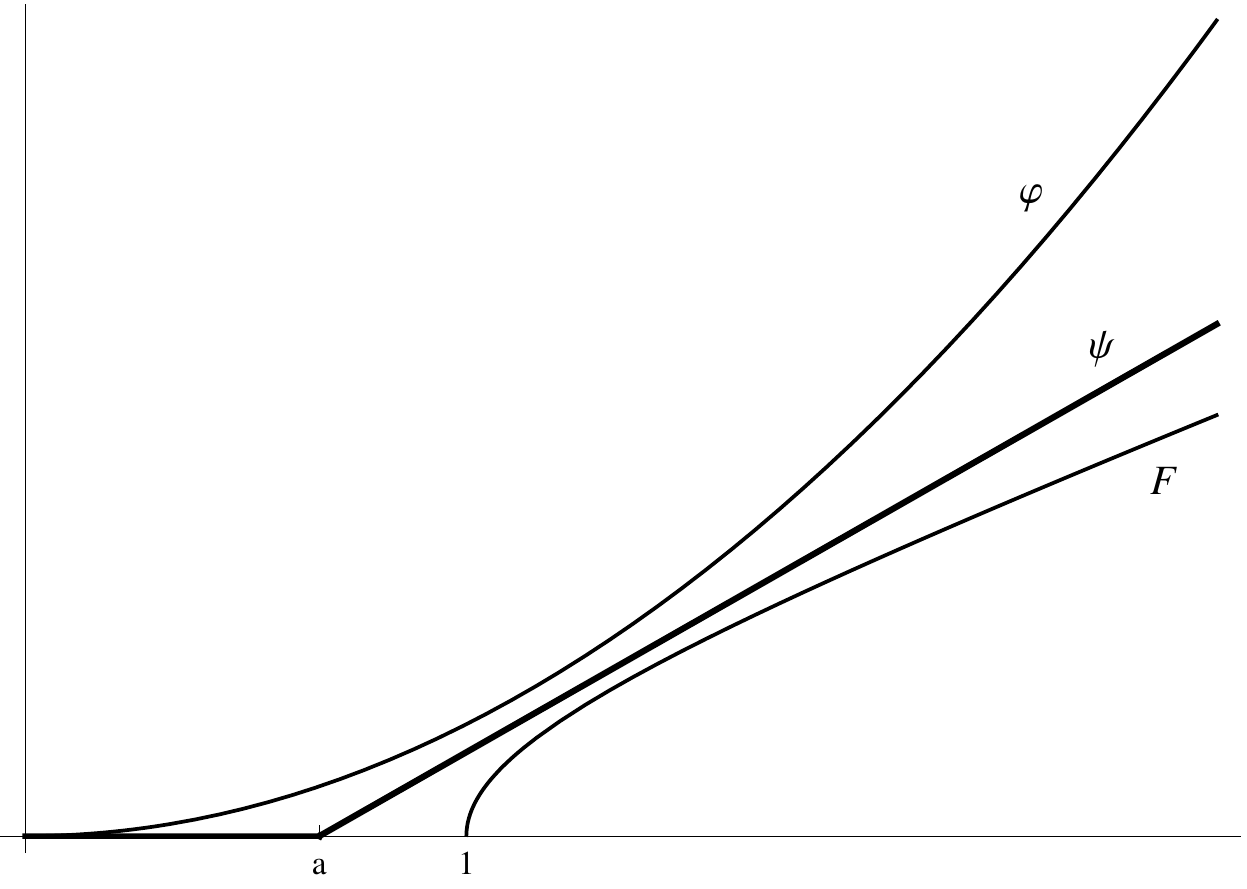}
\par\end{center}

On the one hand, we know that\[
\psi\le\phi=\phip\le\psi^{\circ}=\o_{[0,a^{-1}]}^{\infty}\vee\ell_{\beta^{-1}},\]
 so in particular $\psi\left(\frac{1}{a}\right)\le\psi^{\circ}\left(\frac{1}{a}\right)$.
More explicitly, this means\[
\beta\left(\frac{1}{a}-a\right)\le\frac{1}{\beta}\cdot\frac{1}{a},\]
 or $\beta^{2}\left(1-a^{2}\right)\le1$. 

On the other hand, it is easy to compute that the tangent to $F$
passing through $(a,0)$ has slope $\frac{1}{\sqrt{1-a^{2}}}$. Since
$\psi>F$ we must have $\beta>\frac{1}{\sqrt{1-a^{2}}}$, or $\beta^{2}(1-a^{2})>1$.
This is a contradiction, so no such $\psi$ can exist and $\phi\in T_{q}$
for some $1\le q\le\infty$. 
\end{proof}
Our next goal is to explain how to construct fixed points inside $T_{q}$
for $1<q<\infty$. To do so we will need the following proposition,
which shows that in order to calculate $\phip(x)$ we don't need to
know $\phi(y)$ for all possible values of $y$: 
\begin{prop}
\label{pro:polar-break}Assume $\phi_{1},\phi_{2}\in T_{q}$ for some
$1<q<\infty$.\end{prop}
\begin{enumerate}
\item If $\phi_{1}(x)=\phi_{2}(x)$ for all $x\le q$, then $\phi_{1}^{\circ}(x)=\phi_{2}^{\circ}(x)$
for all $x\ge q$. 
\item If $\phi_{1}(x)=\phi_{2}(x)$ for all $x\ge q$, then $\phi_{1}^{\circ}(x)=\phi_{2}^{\circ}(x)$
for all $x\le q$.\end{enumerate}
\begin{proof}
We will prove (i), and the proof of (ii) is analogous.

Define convex functions $\psi_{L}$ and $\psi_{U}$ by \begin{eqnarray*}
\psi_{U} & = & \phi\vee\o_{[0,q]}^{\infty}\\
\psi_{L} & = & \left(\phi\vee\o_{[0,q]}^{\infty}\right)\wedge\ell_{F'(q),}\end{eqnarray*}
where $\phi$ is either $\phi_{1}$ or $\phi_{2}$ -- the definitions
remain the same regardless of this choice. We claim that $\psi_{L}\le\phi_{i}\le\psi_{U}$
for $i=1,2$. The right inequality is obvious. For the left inequality,
notice that $\phi_{i}$ and $\psi_{L}$ coincide for $x\le q$. For
$x\ge q$ the function $\psi_{L}$ grows linearly, and in fact is
exactly the tangent line to $F$ at the point $q$. Since $\phi_{i}\in T_{q}$
we get that $\psi_{L}$ is also a tangent for $\phi_{i}$, and because
$\phi_{i}$ is convex we must have $\phi_{i}\ge\psi_{L}$:

\noindent \begin{center}
\includegraphics[scale=0.65]{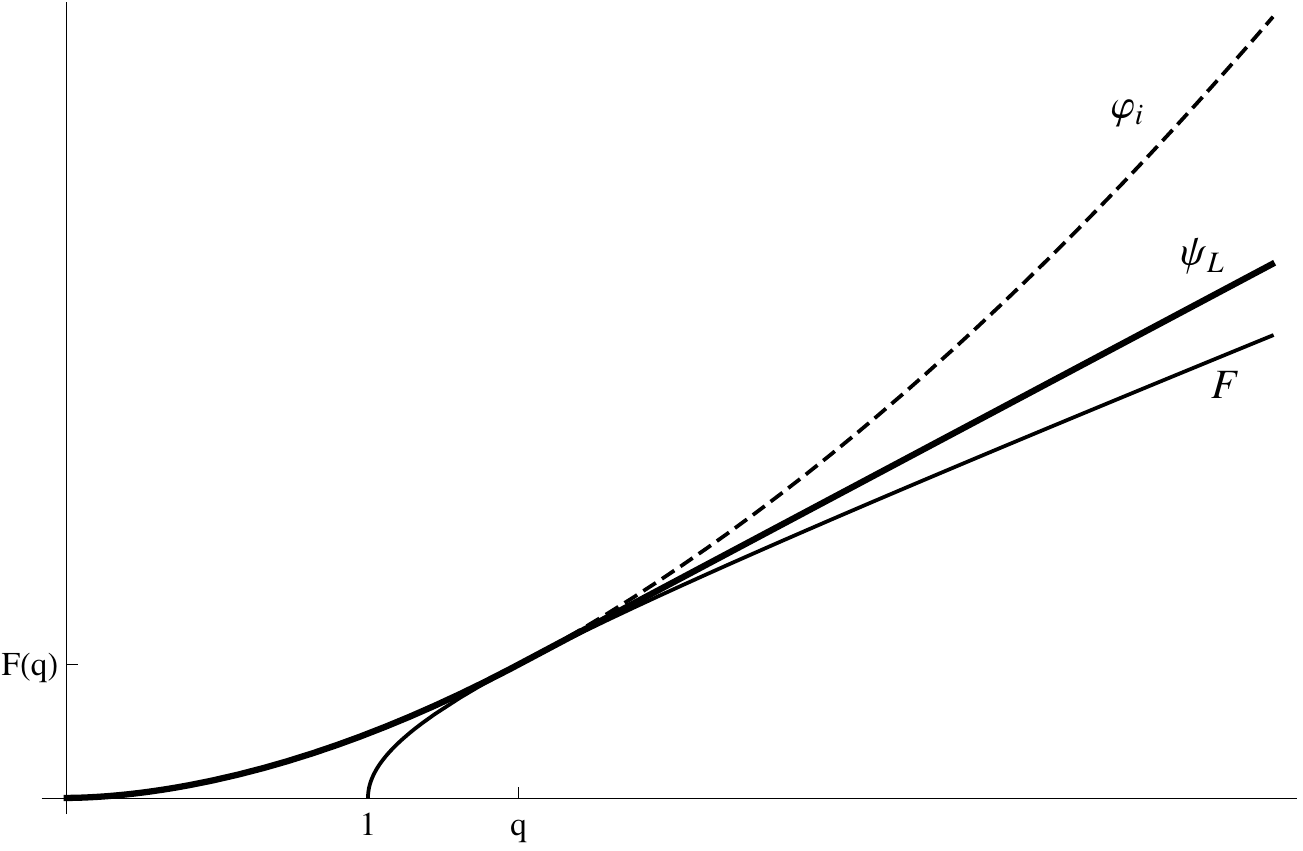}
\par\end{center}

Since polarity is order reversing we get that $\psi_{U}^{\circ}\le\phi_{i}^{\circ}\le\psi_{L}^{\circ}$,
so all we need to show is that $\psi_{U}^{\circ}(x)=\psi_{L}^{\circ}(x)$
for every $x\ge q$. A direct computation yields\begin{eqnarray*}
\psi_{U}^{\circ} & = & \phi_{1}^{\circ}\wedge\o_{[0,q^{-1}]}^{\infty}\\
\psi_{L}^{\circ} & = & \left(\phi_{1}^{\circ}\wedge\o_{[0,q^{-1}]}^{\infty}\right)\vee\ell_{F'(q)^{-1},}\end{eqnarray*}
so we need to show that if $x\ge q$ then\[
\frac{x}{F'(q)}\le\left(\phi_{1}^{\circ}\wedge\o_{[0,q^{-1}]}^{\infty}\right)(x).\]

By Proposition \ref{pro:invariant-set} we know that $\phi_{1}^{\circ}\in T_{q}$.
In particular, the tangent line to $F$ at $q$ is also a tangent
line for $\phi_{1}^{\circ}$. But an easy computation shows that this
line passes through $\left(q^{-1},0\right)$, so we know how $\psi_{U}^{\circ}$
looks:

\noindent \begin{center}
\includegraphics[scale=0.65]{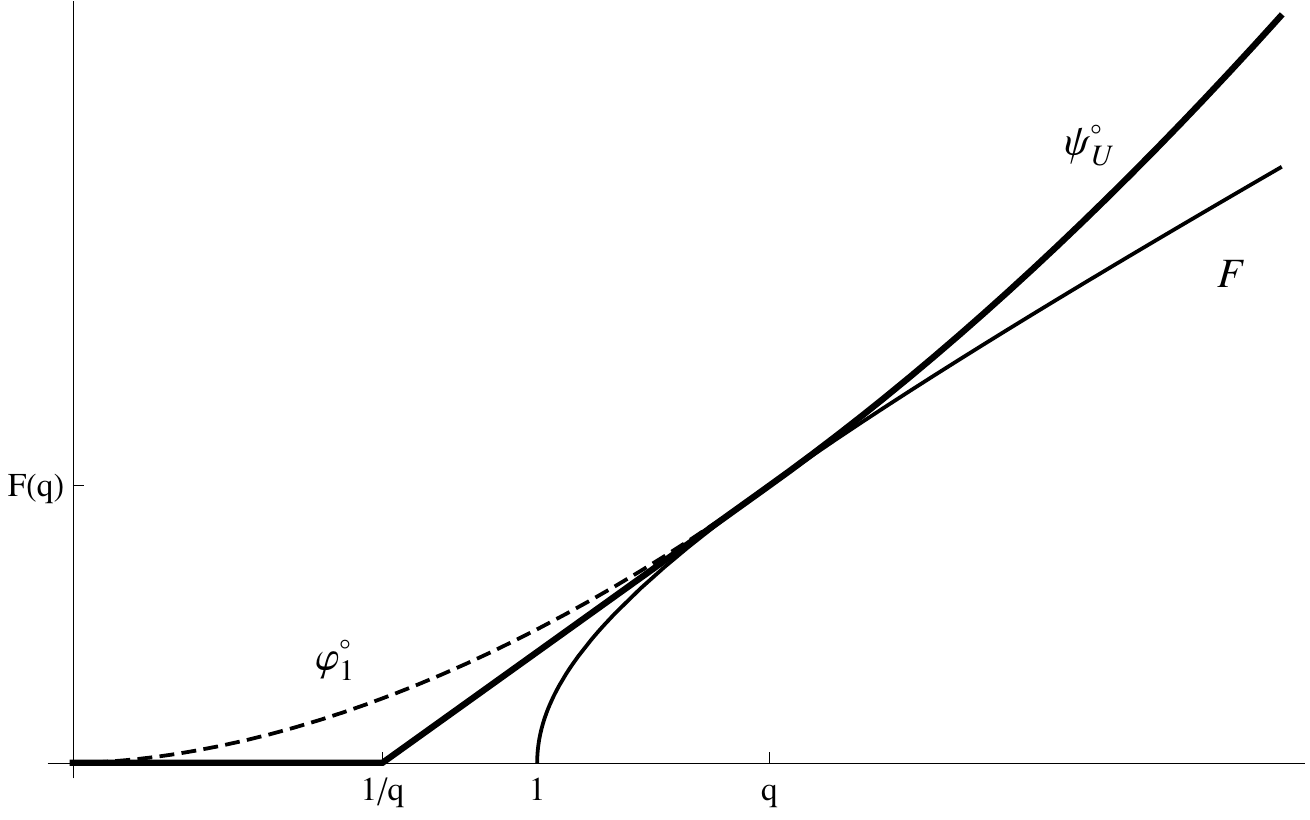}
\par\end{center}

In particular, wee see that if $x\ge q$ then \[
\left(\phi_{1}^{\circ}\wedge\o_{[0,q^{-1}]}^{\infty}\right)(x)=\phi_{1}^{\circ}(x).\]
 Finally, since $\phi_{1}^{\circ}$ is convex, we know that for every
$x\ge q$\[
\frac{\phi_{1}^{\circ}(x)}{x}\ge\frac{\phi_{1}^{\circ}(q)}{q}=\frac{F(q)}{q}=\frac{1}{F'(q)},\]
which implies\[
\frac{x}{F'(q)}\le\phi_{1}^{\circ}(x)\]
 like we wanted. 
\end{proof}
Using the last two propositions it is easy to give a complete characterization
of 1-dimensional self-polar convex functions:
\begin{thm}
\label{thm:main-thm}For every $\phi\in T_{q}$ define \[
\widetilde{\phi}(x)=\begin{cases}
\phi(x) & x\le q\\
\phi^{\circ}(x) & x\ge q.\end{cases}\]
 Then $\widetilde{\phi}$ is self-polar, and any self-polar function
is of the form $\widetilde{\phi}$ for some $\phi$. \end{thm}
\begin{proof}
Since both $\phi$ and $\phip$ are tangent to $F$ at $q$, the function
$\widetilde{\phi}$ is indeed convex. Using Proposition \ref{pro:polar-break}
twice we see that $\widetilde{\phi}$ is self-polar (compare once
with $\phi$, and once with $\phip$). In the other direction, if
$\phi$ is self-polar then by Proposition \ref{pro:invariant-set}
we know that $\phi\in T_{q}$ for some $q$, and then $\phi=\widetilde{\phi}$.
\end{proof}
Finally, we would like to state that, rather surprisingly, there are
self-polar functions in $\cvxn$ which are not rotationally invariant.
As one example, it is easy to compute directly that the function $\phi\in\textrm{Cvx}_{0}\left(\R^{2}\right)$
defined by \[
\phi(x,y)=\begin{cases}
\left|y\right| & \text{if }\left|x\right|\le1\\
\infty & \text{otherwise}\end{cases}\]
 is self-polar. This means that our classification of self-polar functions
in $\cvx$ does not give a complete classification of all self-polar
functions in $\cvxn$. 

\bibliographystyle{plain}
\bibliography{bulletin-paper-revised}

\end{document}